\documentclass[a4paper,10pt]{article}
\usepackage{setspace}
\usepackage{diagrams}
\usepackage{geometry}
\usepackage{epsfig}
\usepackage{amsfonts}
\usepackage{amsthm}
\usepackage{latexsym}
\usepackage{amsmath}
\usepackage{amscd}
\usepackage{amssymb}
\usepackage{enumerate}
\usepackage{graphicx,oldgerm}

\usepackage{authblk}

\theoremstyle{definition}
\newtheorem{thm}{Theorem}[section]
\newtheorem{cor}[thm]{Corollary}
\newtheorem{lem}[thm]{Lemma}

\newtheorem{prop}[thm]{Proposition}
\newtheorem{defi}[thm]{Definition}
\newtheorem{rem}[thm]{Remark}
\newtheorem{note}[thm]{Notation}
\newtheorem{para}[thm]{}

\DeclareMathOperator{\codim}{\mathrm{codim}}
\DeclareMathOperator{\NL}{\mathrm{NL}}
\DeclareMathOperator{\prim}{\mathrm{prim}}

\DeclareMathOperator{\red}{\mathrm{red}}

\DeclareMathOperator{\Hc}{\mathcal{H}om}

\DeclareMathOperator{\Ima}{\mathrm{Im}}

\DeclareMathOperator{\op}{\mathcal{O}_{\mathbb{P}^3}(d)}

\DeclareMathOperator{\p3}{\mathbb{P}^3}

\DeclareMathOperator{\pr}{\mathrm{pr}}

\DeclareMathOperator{\Spec}{\mathrm{Spec}}

\DeclareMathOperator{\N}{\mathcal{N}}
\DeclareMathOperator{\T}{\mathcal{T}}

\DeclareMathOperator{\I}{\mathcal{I}}
\DeclareMathOperator{\mo}{\mathcal{O}}

\newcommand{\mb}[1]{\mathbb{#1}}
\newcommand{\mc}[1]{\mathcal{#1}}
\newcommand{\mr}[1]{\mathrm{#1}}
\newcommand{\ov}[1]{\overline{#1}}
\title{Non-reduced components of the Noether-Lefschetz locus}
\author{Ananyo Dan\thanks{The author has been supported by the DFG under Grant KL-$2244/2-1$\\ \newline Humboldt Universit\"{a}t zu Berlin, Institut f\"{u}r Mathematik, Unter den Linden $6$, Berlin $10099$.\\ e-mail: dan@mathematik.hu-berlin.de}}
\date{\today}

\begin{document}
\maketitle
\begin{abstract}
 In this article we give a criterion for a component of the Noether-Lefschetz locus to be non-reduced. We also produce several examples of such components.
\end{abstract}

\doublespacing

\section{Introduction}
In $1882$, M. Noether claimed the following statement which was later proven by Lefschetz:
For $d \ge 4$, a very general smooth degree $d$ surface $X$ in $\mathbb{P}^3$ has Picard number, 
denoted $\rho(X)$ equal to $1$.
This motivates the definition of the \emph{Noether-Lefschetz locus}, denoted by $\NL_d$ parametrizing the 
space of smooth degree $d$ surfaces $X$ in $\mathbb{P}^3$ with $\rho(X)>1$.
One of the interesting problems is the geometry of the Noether-Lefschetz locus. We know that an 
irreducible component $L$ of $\NL_d$ is an algebraic scheme
\cite{kap} and is not necessarily reduced. In this article we study the scheme structure of $L$ and give a necessary and sufficient condition for non-reducedness of $L$. 
Finally, we give examples of such components.

One of the first results in this direction is due to Green, Griffiths, Voisin and others (\cite{M3, GH, v2}) which states that 
for an irreducible component $L$ of the Noether-Lefschetz locus and $d \ge 4$, the codimension of $L$ in $U_d$ satisfies the
following inequality:
\[d-3 \le \codim (L,U_d) \le \binom{d-1}{3}.\] 
The upper bound follows easily from the fact that $\dim H^{2,0}(X)=\binom{d-1}{3}$ for any $X \in U_d$ (see \cite[\S $6$]{v5}).
We say that $L$ is a \emph{general} component if $\codim L=\binom{d-1}{3}$ and \emph{special} otherwise. 
It was proven by Ciliberto, Harris and Miranda \cite{ca1} that for $d \ge 4$, 
the Noether-Lefschetz locus has infinitely many general components and the union of these components is Zariski dense in $U_d$. 
The guiding principle of much work in the area has been the expectation that special components should be due to the presence of low degree curves.
Voisin \cite{v3} and Green \cite{M3} independently 
proved that for $d \ge 5$, $\codim L=d-3$ if and only if $L$ parametrizes surfaces of degree $d$ containing a line. This component is reduced.
If $d-3<\codim L \le 2d-7$ then $\codim L=2d-7$ and $L$ parametrizes the surfaces containing a conic and is reduced.
Maclean \cite{c1} showed that for $d=5$ there exists a component $L$ of codimension $2d-6$ such that $L_{\red}$, the associated reduced subscheme,
parametrizes the surfaces containing two lines on the same plane 
and this component is non-reduced.
Otwinowska \cite{A1} proved that for an integer $b>0$ and $d \gg b$ if $\codim L \le bd$ then $L$ parametrizes surfaces containing a curve of 
degree at most $b$.

We now briefly discuss the main ideas used in this article. One of the important observations is that an irreducible component of the Noether-Lefschetz locus is a Hodge locus. This is a consequence of the
Lefschetz $(1,1)$-theorem. We recall the notion of the Hodge locus.
Denote by $U_d \subseteq \mathbb{P}(H^0(\mathbb{P}^3, \op))$ 
the open subscheme parametrizing smooth projective hypersurfaces in $\mathbb{P}^3$ of degree $d$.
Let $\mathcal{X} \xrightarrow{\pi} U_d$ be the corresponding universal family. For a given $F \in U_d$, denote by $X_F$ the surface $X_F:=\pi^{-1}(F)$. 
Let $X \in U_d$ and $U \subseteq U_d$ be a simply connected neighbourhood of $X$ in $U_d$ (under the analytic topology).
Then $\pi|_{\pi^{-1}(U)}$ induces a variation of Hodge structure $(\mathcal{H}, \nabla)$ on $U$ where $\mathcal{H}:=R^2\pi_*\mathbb{Z} \otimes 
\mathcal{O}_U$ and $\nabla$ is the Gauss-Manin connection. Note that $\mathcal{H}$ defines a local system on $U$ whose fiber over 
a point $F \in U$ is $H^2(X_F,\mathbb{Z})$ where $X_F=\pi^{-1}(F)$. Consider a non-zero element $\gamma_0 \in H^2(X_F,\mathbb{Z}) \cap H^{1,1}(X_F,\mathbb{C})$
such that $\gamma_0  \not= c_1(\mathcal{O}_{X_F}(k))$ for $k \in \mathbb{Z}_{>0}$. 
This defines a section $\gamma \in (\mathcal{H} \otimes \mathbb{C})(U)$. Let $\overline{\gamma}$ be the image
of $\gamma$ in $\mathcal{H}/F^2(\mathcal{H} \otimes \mathbb{C})$. The Hodge loci, denoted $\NL(\gamma)$ is then defined as
\[ \NL(\gamma):=\{G \in U | \overline{\gamma}_G=0\},\]
where $\overline{\gamma}_G$ denotes the value at $G$ of the section $\overline{\gamma}$. For an irreducible component 
$L \subset \NL_d$ and $X \in L$, general, we can find $\gamma \in H^2(X,\mb{Z}) \cap H^{1,1}(X,\mb{C})$
such that $\overline{\NL(\gamma)}=L$ (the closure taken under Zariski topology on $U_d$). 
The Gauss-Manin connection gives rise to a differential map, $\ov{\nabla}(\gamma):T_XU \to H^2(\mo_X)$, where $T_XU$ is the 
tangent space to $U$ at the point corresponding to $X$.
The tangent space at $X$ to $\NL(\gamma)$ is defined to be $\ker(\ov{\nabla}(\gamma))$.

For a better understanding of the tangent space to $\NL(\gamma)$ we will use the theory of semi-regularity as introduced by Bloch \cite{b1}. 
Recall, given a curve $C$ in a smooth surface $X$ in $\p3$, the semi-regularity map is the morphism $\pi:H^1(\N_{C|X}) \to H^2(\mo_X)$
arising from the short exact sequence,
\[0 \to \mo_X \to \mo_X(C) \to \N_{C|X} \to 0.\] We show in Theorem \ref{hf12}:
\begin{thm}
Let $X$ be a smooth surface in $\p3$, $C$ a curve in $X$ and $\gamma$ the cohomology class of $C$. 
Then, the differential map $\ov{\nabla}(\gamma):H^0(\N_{X|\p3}) \to H^2(\mo_X)$ factors through the semi-regularity map.
\end{thm}
Using this we can compare infinitesimal deformations of the pair $(C,X)$ such that $C$ remains a curve to the
infinitesimal deformations of $X$ such that the cohomology class of $C$ remains a Hodge class.

We then use this to give a necessary and sufficient condition for non-reducedness of an irreducible component of $\NL_d$:
\begin{thm}\label{in01 }
 Consider an irreducible component $L$ of the Noether-Lefschetz locus. Then, for a general element $X \in L$, there exists $\gamma \in H^{1,1}(X,\mb{Z})$ satisfying the 
 following conditions:
 \begin{enumerate}
  \item $\gamma$ is the cohomology of a curve, say $C$ in $X$, such that the corresponding semi-regularity map is injective
  (i.e., $C$ is semi-regular)
  \item There is an irreducible component, say $H_\gamma$ of the flag Hilbert scheme parametrizing pairs $(C'\subset X')$ for 
  $C'$ a curve (resp. $X'$ a surface) with Hilbert polynomial the same as $C$ (resp. $X$) 
   such that $\pr_2(H_\gamma)_{\red}$ is isomorphic to $L_{\red}$ and $\pr_2(T_{(C,X)}H_{\gamma})$
  is the same as $T_X(\NL(\gamma))$.
 \item $H^1(\mo_X(-C)(d))=0$ i.e., for every infinitesimal deformation of $C$ there exists a corresponding infinitesimal deformation of $X$
 containing it.
 \end{enumerate}
Furthermore, $L$ is non-reduced if and only if $\dim \pr_1(H_\gamma) <h^0(\N_{C|\p3})$.
\end{thm}
See Proposition \ref{a81}, Corollary \ref{a82} and Theorem \ref{a4} 
for the precise statements and proofs.
The proofs of the parts $(1)$ and $(3)$ is an application of Serre's vanishing theorem. Part $(2)$ uses the description of
$\ov{\nabla}(\gamma)$ in terms of the semi-regularity map as given in Theorem \ref{hf12}.

In order to give examples in the case $L$ is non-reduced, we fix an integer $d \ge 5$ and consider certain irreducible non-reduced components of the 
Hilbert scheme of curves in $\p3$ such that the general element $C$ is contained in a smooth degree $d$ surface, $X$.
We show in Lemma \ref{hf11} and Theorem \ref{cas12} that if $d-4 \ge \deg(C)$ then $\mo_X(-C)$ is $d$-regular in the sense of Castelnuovo-Mumford. 
This tells us that given an infinitesimal deformation of such $C$ there exists a corresponding infinitesimal deformation of 
$X$ containing it. Using this we show that the Hodge locus corresponding to the cohomology class of $C$ is non-reduced. Among other examples we show that,
\begin{thm}
 Let $d \ge 5$, $X$ a surface containing two distinct coplanar lines, say $l_1, l_2$, $C$ a divisor in $X$ of the form $2l_1+l_2$ and $\gamma$ the cohomology class of $C$.
 Then, $\ov{\NL(\gamma)}$ (closure taken in $U_d$ under Zariski topology) is non-reduced.
\end{thm}


{\bf{Acknowledgement:}} I would like to thank R. Kloosterman for introducing me to the topic, reading the preliminary version of this article and several helpful discussions.
I would also like to thank V. Srinivas for useful suggestions and to N. Tarasca for reading parts of the article.

\section{Introduction to Noether-Lefschetz locus}

\begin{para}
 In this section we recall the basic definitions of Noether-Lefschetz locus. See \cite[\S $9, 10$]{v4} and \cite[\S $5, 6$]{v5} for a detailed presentation of the subject.
\end{para}

\begin{defi}
 Recall, for a fixed integer $d \ge 5$, the \emph{Noether-Lefschetz locus}, denoted $\NL_d$, parametrizes the space of smooth degree $d$ surfaces
 in $\p3$ with Picard number greater than $1$. Using the Lefschetz $(1,1)$-theorem this is the parametrizing space of smooth degree $d$ surfaces with 
 $H^{1,1}(X,\mb{C}) \cap H^2(X,\mb{Z}) \not= \mb{Z}$. 
\end{defi}

\begin{note}\label{nl1}
 Let $X \in U_d$ and $\mo_X(1)$, the very ample line bundle on $X$ determined by the closed immersion $X \hookrightarrow \p3$ arising (as in \cite[II.Ex.$2.14$(b)]{R1}) from the graded homomorphism
of graded rings $S \to S/(F_X)$, where $S=\Gamma_*(\mo_{\p3})$ and $F_X$ is the defining equations of $X$.
Denote by $H_X$ the very ample line bundle $\mo_X(1)$. 
\end{note}

     \begin{note}
     Let $X$ be a surface. Denote by $H^2(X,\mb{C})_{\prim}$, the primitive cohomology. There is a natural
      projection map from $H^{2}(X,\mb{C})$ to $H^{2}(X,\mb{C})_{\prim}$. For $\gamma \in H^{2}(X,\mb{C})$, denote by $\gamma_{\prim}$ the image of $\gamma$ under this morphism.
      Since the very ample line bundle $H_X$ remains of type $(1,1)$ in the family $\mc{X}$, we can conclude that
        $\gamma \in H^{1,1}(X)$ remains of type $(1,1)$ if and only if $\gamma_{\prim}$ remains of type $(1,1)$. In particular,
      $\NL(\gamma)=\NL(\gamma_{\prim})$.
     \end{note}

\begin{para}
 Note that, $\NL_d$ is a countable union of subvarieties. Every irreducible component of $\NL_d$ is locally of the 
 form $\NL(\gamma)$ for some $\gamma \in H^{1,1}(X) \cap H^2(X,\mb{Z})$, $X \in \NL_d$ such that $\gamma_{\prim} \not=0$. 
 \begin{lem}[{\cite[Lemma $5.13$]{v5}}]
 There is a natural analytic scheme structure on $\ov{\NL(\gamma)}$ (closure in $U_d$ under Zariski topology).
\end{lem}
\end{para}

\begin{note}
 Let $X_1$ be a projective scheme, $X_2 \subset X_1$, a closed subscheme. Denote by $\N_{X_2|X_1}$ the normal sheaf $\Hc_{X_1}(\I_{X_2/X_1},\mo_{X_2})$, where $\I_{X_2/X_1}$ is the ideal sheaf of $X_2$ in $X_1$.
\end{note}

\begin{defi} 
We now discuss the tangent space to the Hodge locus, $\NL(\gamma)$.
We know that the tangent space to $U$ at $X$, $T_XU$ is isomorphic to $H^0(\N_{X|\p3})$.
This is because $U$ is an open subscheme of the Hilbert scheme $H_{Q_d}$, the tangent space of which at the point $X$ is simply
$H^0(\N_{X|\p3})$. 
Given the variation of Hodge structure above, we have (by Griffith's transversality) the differential map:
\[\overline{\nabla}:H^{1,1}(X) \to \mathrm{Hom}(T_XU,H^2(X,\mathcal{O}_X))\]induced by the Gauss-Manin connection.
Given $\gamma \in H^{1,1}(X)$ this induces a morphism, denoted $\overline{\nabla}(\gamma)$ from $T_XU$ to $H^2(\mo_X)$.
\end{defi}

\begin{lem}[{\cite[Lemma $5.16$]{v5}}]\label{tan1}
The tangent space at $X$ to $\NL(\gamma)$ equals $\ker(\overline{\nabla}(\gamma))$.
\end{lem}

\begin{defi}\label{gr0}
 The boundary map $\rho$, from $H^0(\N_{X|\p3})$ to $H^1(\T_X)$ arising from the long exact sequence associated to the 
 short exact sequence:
 \begin{equation}\label{ex9}
 0 \to \T_X \to \T_{\p3}|_X \to \N_{X|\p3} \to 0
 \end{equation}
  is called the \emph{Kodaira-Spencer} map. The morphism $\overline{\nabla}(\gamma)$ is related to the Kodaira-Spencer map as follows:
\end{defi}

\begin{thm}\label{a3}
The differential map $\overline{\nabla}(\gamma)$ conincides with the following:
\[T_XU\cong H^0(\N_{X|\p3}) \xrightarrow{\rho} H^1(\T_X) \xrightarrow{\bigcup \gamma} H^2(\mo_X)\]
and under the identification $\N_{X|\p3} \cong \mo_X(d)$, $\ker(\rho) \cong J_d(F)$, where $J_d(F)$ denotes the degree $d$ graded piece of the Jacobian ideal of $X$.
\end{thm}

\begin{proof}
 See \cite[Theorem $10.21$]{v4} and \cite[Lemma $6.15$]{v5} for a proof. 
\end{proof}

\section{Introduction to flag Hilbert schemes}

We briefly recall the basic definition of flag Hilbert schemes and its tangent space in the relevant case. See \cite[\S $4.5$]{S1}
for further details.

\begin{defi}
Given an $m$-tuple of numerical polynomials $\mathcal{P}(t)=(P_1(t),P_2(t),...,P_m(t))$, we define the  contravariant functor, called the \emph{Hilbert flag functor} 
relative to $\mathcal{P}(t)$,
\begin{eqnarray*}
FH_{\mathcal{P}(t)}:(\mbox{schemes}) &\to& \mbox{sets}\\
S &\mapsto& \left\{\begin{tabular}{l|l}
$(X_1,X_2,...,X_m)$& $X_1 \subset X_2 \subset ... \subset X_m \subset \mathbb{P}^3_S$\\
& $X_i \mbox{ are } S-\mbox{flat with Hilbert polynomial } P_i(t)$
\end{tabular}\right\}
\end{eqnarray*}
We call such an $m$-\emph{tuple a flag relative to} $\mathcal{P}(t)$. The functor is representable by a projective scheme $H_{\mc{P}(t)}$, called \emph{flag Hilbert scheme}.
\end{defi}

\begin{defi}\label{hf1}
In the case $m=2$, we have the following definition of the tangent space at a pair $(X_1,X_2)$ to the flag Hilbert scheme $H_{P_1,P_2}$:
\begin{equation}\label{dia2}
\begin{diagram}
&T_{(X_1,X_2)}H_{P_1,P_2} &\rTo &H^0(\N_{X_2|\mathbb{P}^3}) & \\
&\dTo & \square &\dTo & \\
&H^0(\N_{X_1|\mathbb{P}^3}) &\rTo &H^0(\N_{X_2|\mathbb{P}^3} \otimes \mathcal{O}_{X_1})
\end{diagram}
\end{equation} 
\end{defi}

\begin{defi}\label{n1}
We now fix certain notation for the rest of this article.
By a \emph{component} of $\NL_d$, we mean an irreducible component. By a \emph{surface} we always mean a smooth surface in $\mathbb{P}^3$.
Denote by $Q_d$ the Hilbert polynomial of degree $d$ surfaces in $\p3$. Given, a Hilbert polynomial $P$, denote by $H_P$ the corresponding Hilbert scheme and by $H_{P,Q_d}$ the corresponding flag Hilbert scheme.
\end{defi}

\section{Hodge locus and semi-regularity}

\begin{para}
In this section, we consider the case when $\gamma$ is the cohomology class of a curve $C$ in a smooth degree $d$ surface $X$ in $\p3$.
We see that the differential map $\ov{\nabla}(\gamma)$ factors through the semi-regularity map, $H^1(\N_{C|X}) \to H^2(\mo_X)$.
Finally, using this description, we compute $\ker \ov{\nabla}(\gamma)$, which is the tangent space to the Hodge locus, $\NL(\gamma)$.
 \end{para}

\begin{para}
We start with the definition of a semi-regular curve.
 Let $X$ be a surface and $C \subset X$, a curve in $X$. Since $X$ is smooth, $C$ is local complete intersection in $X$.
Denote by $i$ the closed immersion of $C$ into $X$. This gives rise to the short exact sequence:
 \begin{equation}\label{sh1}
 0 \to \mo_X \to \mo_X(C) \to \N_{C|X} \to 0.
                 \end{equation}                       
 The \emph{semi-regularity map} $\pi$ is the boundary map from $H^1(\N_{C|X})$ to $H^2(\mo_X)$.
 We say that $C$ is \emph{semi-regular} if $\pi$ is injective. 
\end{para}



\begin{para}\label{hf23}
 Let $X$ be a smooth surface and $C$ a local complete intersection curve in $X$.
 Consider the natural differential map, $d_X:\mo_X(-C)  \otimes \mo_C \to \Omega^1_X  \otimes \mo_C$.
 This yields a map $\mr{Ext}^1_C(\Omega^1_X  \otimes \mo_C,\mo_C) \to \mr{Ext}^1_C(\mo_X(-C)  \otimes \mo_C,\mo_C)$.
 Since $\Omega^1_X  \otimes \mo_C$ and $\mo_X(-C)  \otimes \mo_C$ are locally free $\mo_C$-modules, \cite[III.Ex. $6.5$]{R1} implies that 
 \[\mc{E}xt^1_C(\Omega^1_X  \otimes \mo_C,\mo_C)=0=\mc{E}xt^1_C(\mo_X(-C)  \otimes \mo_C,\mo_C).\]
 Using the local to global Ext spectral sequence, we can conclude that  
 $\mr{Ext}^1_C(\Omega^1_X  \otimes \mo_C,\mo_C)$ (resp. $\mr{Ext}^1_C(\mo_X(-C)  \otimes \mo_C,\mo_C)$) is isomorphic to 
 $H^1(\T_X  \otimes \mo_C)$ (resp. $H^1(\N_{C|X})$) where $\T_X$ is the tangent sheaf on $X$ and $\N_{C|X}$ is the 
 normal sheaf of $C$ in $X$. 
 Let $u_*:H^1(\T_X) \to H^1(\N_{C|X})$ be the composition of the restriction morphism $j_3:H^1(\T_X) \to H^1(\T_X \otimes \mo_C)$
 with the morphism $j_1:H^1(\T_X \otimes \mo_C) \to H^1(\N_{C|X})$ obtained above.
\end{para}

\begin{thm}[{\cite[Theorem $4.5, 5.5$]{fle}}]\label{hf4}
Let $C, X$ be as in \ref{hf23} and $[C]$ denote the cohomology class of $C$.
 The morphism $u_*$ defined above satisfies the following commutative diagram:
 \[\begin{diagram}
    H^1(\T_X)&\rTo^{u_*}&H^1(\N_{C|X})\\
    \dTo^{\bigcup [C]}&\ldTo^{\pi}\\
    H^2(\mo_X)
   \end{diagram}\]
   where $\bigcup [C]$ is the morphism which takes $\xi$ to $\xi \bigcup [C]$, the cup-product of $\xi$ and $[C]$.
\end{thm}

\begin{rem}
 One can see that $u_*$ is the obstruction map in the sense, for $\xi \in H^1(\T_X)$, $u_*(\xi)=0$ if and only if $C$ lifts to a local complete intersection in the infinitesimal deformation of $\xi$ (see \cite[Proposition $2.6$]{b1}).
 We will now see a diagram (in Theorem \ref{hf12}) which illustrates this fact more clearly.
\end{rem}

\begin{para}
Recall, the following short exact sequence of normal sheaves:
 \begin{equation}\label{sh2}
  0 \to \N_{C|X} \to \N_{C|\p3} \to \N_{X|\p3}  \otimes \mo_C \to 0
 \end{equation}
 which arises from the short exact sequence:
 \begin{equation}\label{ex0c}
 0 \to \I_X \to \I_C \xrightarrow{j^\#} j_*\mo_X(-C) \to 0.
 \end{equation}
after applying $\Hc_{\p3}(-,j_{0_*} \mo_C)$, where $j_0$ is the closed immersion of $C$ into $\p3$.
\end{para}

The following lemma will be used in the proof of Theorem \ref{hf12} below.

\begin{lem}\label{de1}
Given $j:C \hookrightarrow X$ the closed immersion, 
 the following diagrams are commutative and the horizontal rows are exact:
 \[\begin{diagram}
    0 &\rTo &\I_X  \otimes \mo_C &\rTo &\I_C  \otimes \mo_C &\rTo^u &\mo_X(-C) \otimes \mo_C &\rTo&0\\
    & &\dTo^{\cong}&\circlearrowright&\dInto^{d_{\p3}}&\circlearrowright&\dInto^{d_X}& \\
    0 &\rTo &\I_X  \otimes \mo_C &\rTo^{d_{\p3}} &\Omega^1_{{\p3}}  \otimes \mo_C &\rTo^v &\Omega^1_X  \otimes \mo_C &\rTo &0
   \end{diagram}\]
   where $d_{\p3}$ (resp. $d_X$) are the K\"{a}hler differential operator, $u$ is induced by $j^\#$ and $v$ 
   maps $d_{\p3}f \otimes g$ to $gd_Xj^{\#}(f)$ on open sets.
 \end{lem}

 \begin{proof}
The exactness of the lower horizontal sequence is explained in \cite[Theorem $25.2$]{mat}. 
We now show the exactness of the upper horizontal sequence. The sequence is obtained via pulling back (\ref{ex0c}) 
by the closed immersion of $C$ into $\p3$.
Since tensor product is right exact, it remains to prove that the morphim from $\I_X  \otimes \mo_C$ to $\I_C  \otimes \mo_C$
is injective. It suffices to prove this statement on the stalk  for all $x \in C$, closed point. Note that,
$\I_{X,x}  \otimes \mo_{C,x}$ (resp. $\I_{C,x}  \otimes \mo_{C,x}$) is isomorphic to $\I_{X,x}/\I_{X,x}\I_{C,x}$ (resp.
$\I_{C,x}/\I_{C,x}^2$). So, we need to show that $\I_{X,x} \cap \I_{C,x}^2$
is contained in $\I_{X,x}.\I_{C,x}$.

Since $C$ is a local complete intersection curve, $\I_{C,x}$ is generated by a regular sequence, say $(f_x, g_x)$ 
 where $g_x$ generates the ideal $\I_{X,x}$. So, $\I_{X,x}.\I_{C,x}=(g_x^2,f_xg_x), \I_{C,x}^2=(f_x^2,f_xg_x,g_x^2)$. Note that any element in $(g_x) \cap (f_x^2,f_xg_x,g_x^2)$ is divisible by $f_x^2$
modulo the ideal $(g_x^2,f_xg_x)$. Therefore it is divisible by $g_xf_x^2$, hence is an element in $\I_{X,x}.\I_{C,x}$. It directly follows that the natural morphism from 
$\I_{X,x}/\I_{X,x}.\I_{C,x}$ to $\I_{C,x}/\I_{C,x}^2$ is injective.


The commutativity of the diagrams follows directly from the description of the relevant morphisms.
\end{proof}


\begin{thm}\label{hf12}
Let $X$ be a smooth surface, $C \subset X$ and $\gamma=[C] \in H^{1,1}(X,\mb{Z})$.
 We then have the following commutative diagram
 \[\begin{diagram}
  & &  & &T_{(C,X)}H_{P,Q_d}&\rTo&H^0(X,\N_{X|\p3})&\rTo^{\ov{\nabla}(\gamma)}&H^2(X,\mo_X)\\
  & &  & &\dTo&\square&\dTo^{\rho_C}&\circlearrowright&\uTo^{\pi_C}& \\
    0&\rTo&H^0(C,\N_{C|X})&\rTo^{\phi_C}&H^0(C,\N_{C|\p3})&\rTo^{\beta_C}&H^0(C,\N_{X|\p3} \otimes \mo_C)&\rTo^{\delta_C}&H^1(C,\N_{C|X}) 
       \end{diagram}\]
where the horizontal exact sequence comes from (\ref{sh2}), $\pi_C$ is the semi-regularity map and $\rho_C$ is the natural
pull-back morphism.
\end{thm}

\begin{proof}
The only thing left to prove is that $\ov{\nabla}(\gamma)$ is the same as $\pi_C \circ \delta_C \circ \rho_C$.
Using Theorems \ref{a3} and \ref{hf4}, we have that $\ov{\nabla}(\gamma)$ factors as:
\[H^0(\N_{X|\p3}) \xrightarrow{\rho} H^1(\T_X) \xrightarrow{u_*} H^1(\N_{C|X}) \xrightarrow{\pi_C} H^1(\mo_X).\]
Hence it suffices to show that $u_* \circ \rho$ is the same as $\delta_C \circ \rho_C$.

Recall,  under the notations as in \ref{hf23} we can factor $u_*$ as $j_1 \circ j_3$. Hence, it suffices to construct
a morphism $j_2:H^0(\N_{X|\p3} \otimes \mo_C) \to H^1(\T_X \otimes \mo_C)$ such that following two diagrams commute:
\begin{equation}\label{dc1}
\begin{diagram}
   H^0(\N_{X|\p3})&\rTo^{\rho}&H^1(\T_X)&\rTo^{j_3}&H^1(\T_X \otimes \mo_C)&\rTo^{j_1}&H^1(\N_{C|X})\\
   &\rdTo^{\rho_C} & &\ruDashto^{j_2}& &\ruTo(4,2)^{\delta_C}\\
   &  &H^0(\N_{X|\p3} \otimes \mo_C)
  \end{diagram}
  \end{equation}

We define $j_2$ in the following way: Take the following commutative diagram of short exact sequences:
\begin{equation}\label{com6}
\begin{diagram}
0 &\rTo &\T_X &\rTo &\T_{\p3} \otimes \mo_X &\rTo &\N_{X|\p3} &\rTo &0\\
& &\dTo&\circlearrowright&\dTo&\circlearrowright&\dTo\\
0 &\rTo &\T_X \otimes \mo_C &\rTo &\T_{\p3} \otimes \mo_C &\rTo &\N_{X|\p3} \otimes \mo_C&\rTo &0
\end{diagram}
\end{equation}
Then $j_2$ arises from the associated long exact sequence:
 \[\begin{diagram}
   H^0(\N_{X|\p3})& \rTo^{\rho} &H^1(X,\T_X)\\
   \dTo^{\rho_C}&\circlearrowright &\dTo^{j_3}\\
   H^0(\N_{X|\p3}  \otimes \mo_C)& \rTo^{j_2} &H^1(\T_X  \otimes \mo_C)
  \end{diagram}\]
  where the maps (other than $j_2$) is the same as defined above. This gives the commutativity of the first
  square in the diagram (\ref{dc1}).
  
  We now prove the commutativity of the second diagram.
Since the terms in the short exact sequences in Lemma \ref{de1} are locally free $\mo_C$-modules, we get the 
dual diagram of short exact sequence by applying $\Hc_C(-,\mo_C)$ to it. This gives us the following: 
\[\begin{diagram}
0 &\rTo &\N_{C|X} &\rTo &\N_{C|\p3} &\rTo &\N_{X|\p3} \otimes \mo_C &\rTo &0\\
& &\uTo&\circlearrowright&\uTo&\circlearrowright&\uTo^{\cong}\\
0 &\rTo &\T_X \otimes \mo_C &\rTo &\T_{\p3} \otimes \mo_C &\rTo &\N_{X|\p3} \otimes \mo_C&\rTo &0
\end{diagram}\]
where the bottom short exact sequence is the same as in (\ref{com6}).
Taking the associated long exact sequence,
we get 
 \[\begin{diagram}
    & &H^0(\N_{X|\p3} \otimes \mo_C)\\
    & \ldTo^{j_2} & \dTo^{\circlearrowright\, \, \,}_{\delta_C}\\
    H^1(\T_X \otimes \mo_C)&\rTo^{j_1}&H^1(\N_{C|X})
   \end{diagram}\]
The theorem follows.
\end{proof}

\begin{cor}\label{hf12c}
 Let $X$ be a smooth degree $d$ surface in $\p3$, $C \subset X$ be a curve with Hilbert polynomial, say $P$ and $[C]$ its corresponding cohomology class. 
Then, the tangent space, \[T_X(\NL([C])) \subset \rho_C^{-1}(\Ima \beta_C)= \pr_2T_{(C,X)}H_{P,Q_d}.\]
Furthermore, if $C$ is semi-regular then we have equality
$T_X(\NL([C]))=\pr_2T_{(C,X)}H_{P,Q_d}$.
\end{cor}

\begin{proof}
 The first statement follows directly from the diagram in Theorem \ref{hf12}.
Now, if $C$ is semi-regular then $\pi_C$ (as in the diagram) is injective. Hence, $\ker \overline{\nabla}([C])=\ker(\delta_C \circ \rho_C)=\rho_C^{-1}(\Ima \beta_C)$. The corollary then follows. 
\end{proof}

  \begin{cor}\label{dim4}
  Notations as in Theorem \ref{hf12}.
   The kernel of $\rho_C$ is isomorphic to $H^0(\mo_X(-C)(d))$ and $\rho_C$ is surjective if and only if $H^1(\mo_X(-C)(d))=0$.
   Moreover, if $H^1(\mo_X(-C)(d))=0$ then $\pr_1(T_{(C,X)}H_{P,Q_d})=H^0(\N_{C|\p3})$.
  \end{cor}

  \begin{proof}
   Since $\N_{X|\p3} \cong \mo_X(d)$ the first statement follows from the short exact sequence,
   \begin{equation}
 0 \to \mo_X(-C)(d) \to \mo_X(d) \to i_*\mo_C(d) \to 0
\end{equation}
and the fact that $H^1(\mo_X(d))=0$,
where $i$ is the closed immersion of $C$ into $X$. 

The last statement follows directly from the definition of $T_{(C,X)}H_{P,Q_d}$ given in \ref{hf1}.
     \end{proof}


\section{Criterion for Non-reducedness}

\begin{para}
In this section we demonstrate the relation between Hodge locus and certain flag Hilbert schemes.
Using this we give a criterion for non-reducedness of components of the Noether-Lefschetz locus.
\end{para}

 We first prove a result that would help
 us determine when a curve is semi-regular.

\begin{lem}\label{hf11}
Let $C$ be a connected reduced curve and $X$ a smooth degree $d$ surface containing $C$. Then, $H^1(\mo_X(-C)(k))=0$ for all $k \ge \deg(C)$. In particular, of $d \ge \deg(C)+4$ then $h^1(\mathcal{O}_X(C))=0$, hence $C$ is semi-regular.
\end{lem}

\begin{proof}
Since $X$ is a hypersurface in $\p3$ of degree $d$, $\I_X \cong \mo_{\p3}(-d)$.  Consider the short exact sequence:
\[0 \to \I_X \to \I_C \to \mo_X(-C) \to 0.\]
We get the following terms in the associated long exact sequence:
\[... \to H^1(\I_C(k)) \to H^1(\mo_X(-C)(k)) \to H^2(\I_X(k)) \to ...\]
Now, $H^2(\mo_{\p3}(k-d))=0$ (see \cite[Theorem $5.1$]{R1}). Now, $\I_{C}$ is $\deg(C)$-regular (see 
\cite[Main Theorem]{gi1}). 
So, $H^1(\I_C(k))=0$ for $k \ge \deg(C)$. This implies $H^1(\mo_X(-C)(k))=0$ for $k \ge \deg(C)$. 
By Serre duality, $0=H^1(\mo_X(-C)(d-4)) \cong H^1(\mo_X(C))$. So, $C$ is semi-regular.
\end{proof}

Recall the following result,
\begin{lem}\label{a4e}
Let $d \ge 5$ and $C$ be an effective divisor on a smooth degree $d$ surface $X$ of the form $\sum_i a_iC_i$ where $C_i$ are integral curves with $\deg(C)+2 \le d$.
Then, $h^0(\N_{C|X})=0.$ In particular, $\dim |C|=0$ where $|C|$ is the linear system associated to $C$.
\end{lem}

\begin{proof}
See \cite[Lemma $3.6$]{D3}.
\end{proof}


\begin{lem}\label{dim3}
 Let $X$ be a smooth surface and $C$ a local complete intersection in $X$. Then, $h^0(\mo_X(-C)(d))=h^0(\I_C(d))-1$ and $h^0(\N_{C|X})=h^0(\mo_X(C))-1$.
\end{lem}

\begin{proof}
 The first equality follows from the short exact sequence,
 \[0 \to \I_X(d) \to \I_C(d) \to i_*\mo_X(-C)(d) \to 0\]
 and the fact that $\I_X \cong \mo_{\p3}$, where $i:X \to \p3$ is the natural closed immersion.
 
 The second equality follows from the short exact sequence (\ref{sh1}) after using the facts $h^0(\mo_X)=1$ and $h^1(\mo_X)=0$.
\end{proof}

We now use these three lemmas given above to produce an explicit formula for the tangent space to the Hodge locus corresponding to the cohomology class of a reduced connected curve.

\begin{prop}\label{hf21}
Let $X$ be a smooth degree $d$ surface in $\p3$ and $C \subset X$,  a reduced, connected curve satisfying $\deg(X)\ge \deg(C)+4$.
Then \[\dim T_X(\NL([C]))=h^0(\I_C(d))-1+h^0(\N_{C|\p3}).\]
\end{prop}

\begin{proof}
Lemma \ref{hf11} tells us that if $\deg(X)\ge \deg(C)+4$ and $C$ reduced, connected then $h^1(\mo_X(-C)(d))=0$ and $h^1(\mo_X(C))=0$. 
Corollary \ref{dim4} implies that $\rho_C$ is surjective.
Lemma \ref{a4e} implies that $h^0(\N_{C|X})=0$, so $\beta_C$ is injective. Using Lemma \ref{dim3} and Corollary \ref{hf12c} we can finally conclude that,
\[\dim T_X(\NL([C]))=h^0(\mo_X(-C)(d))+h^0(\N_{C|\p3})-h^0(\N_{C|X})=h^0(\I_C(d))-1+h^0(\N_{C|\p3}).\]This proves the proposition.
\end{proof}

We recall a result due to Bloch which gives a relation between the Hodge locus corresponding to the cohomology class of a semi-regular curve and the corresponding flag Hilbert scheme.
\begin{thm}[{\cite[Theorem $7.1$]{b1}}]\label{b71}
 Let $X \to S$ be a smooth projective morphism with $S=\Spec(\mb{C}[[t_1,...,t_r]])$. Let $X_0 \subset X$ be the
 closed fiber and let $Z_0 \subset X_0$ be a local complete intersection subscheme of codimension $1$. Suppose that the
 topological cycle class $[Z_0] \in H^2(X_0.\mb{Z})$ lifts to a horizontal class $z \in F^1H^2(X)$ (where $F^\bullet$
 refers to the Hodge filtration) and that $Z_0$ is semi-regular in $X_0$. Then, $Z_0$ lifts to a subscheme $Z \subset X$.
\end{thm}

This result implies the following:

\begin{thm}\label{dim1}
 Let $X$ be a surface, $C$ be a semi-regular curve in $X$ and $\gamma \in H^{1,1}(X,\mb{Z})$ be the class of  $C$. 
 For any irreducible component $L'$ of $\overline{\NL(\gamma)}$
 (the closure is taken in the Zariski topology on $U_d$)
 there exists an irreducible component $H'$ of $H_{P,Q_{d_{\red}}}$ containing the pair $(C,X)$ such that  
 $\pr_2(H')$  coincides with $L'_{\red}$, the associated reduced subscheme, where $\pr_2$ is  the second projection map from $H_{P,Q_d}$ to $H_{Q_d}$. 
\end{thm}

\begin{proof}
Using basic deformation theory, one can check that the image under $\pr_2:H_{P,Q_d} \to H_{Q_d}$ of all the irreducible components of $H_{{P,Q_d}_{\red}}$ containing the pair $(C,X)$
is contained in $\ov{\NL(\gamma)}$. But, Theorem \ref{b71} proves the converse i.e., $\ov{\NL(\gamma)}_{\red}$ is contained in 
$\pr_2(H_{P,Q_d})$. This proves the theorem.
\end{proof}

The following result gives a geometric classification of the components of the Noether-Lefschetz locus:
\begin{prop}\label{a81}
 Consider an irreducible component $L$ of the Noether-Lefschetz locus. Then, for a general element $X \in L$, there exists $\gamma \in H^{1,1}(X,\mb{Z})$ satisfying:
 $\gamma$ is the cohomology class of a semi-regular curve, say $C$, and $\ov{\NL(\gamma)} \cong L$. 
\end{prop}

\begin{proof}
 By definition, $L$ is locally a Hodge locus. So, there exists $\gamma \in H^{1,1}(X,\mb{Z})$ such that $L$ is 
 locally of the form $\NL(\gamma)$. 
 Consider $\gamma$ as an element in the Neron-Severi group. Since $\NL(\gamma) \cong \NL(a\gamma)$ for any 
 integer $a$, we can assume that 
 $\gamma$ is of the form $\sum_i a_i[C_i]$ where $a_i \in \mb{Z}$ and $C_i$ are irreducible curves and
 $\gamma.H_X>0$ (otherwise replace $\gamma$ by $-\gamma$).
 Since $H_X^2=d$, $(\gamma+nH_X)^2>0$ for $n\gg0$. 
 Using \cite[Corollary V.$1.8$]{R1}, there exists a local complete intersection 
 curve $C$ in $X$ such that  $C$ is linearly equivalent to $n'(\gamma+nH_X)$ for $n' \gg 0$. Since $\NL(\gamma)=\NL(a\gamma+bH_X)$ for any $a, b$ non-zero, we have $\NL(\gamma) = \NL([C])$. 
  
 Now, Serre's vanishing theorem and Serre duality implies that for $m \gg 0$, $H^1(\mo_X(C)(m))=0$.
 For such $m$, take $C' \in |C+mH_X|$ and $\gamma$ to be the cohomology class of $C'$. Then, 
 $\ov{\NL(\gamma)}$ is isomorphic to $L$ and $C'$ is semi-regular.
\end{proof}

\begin{cor}\label{a82}
 Let $L$ be an irreducible component of the Noether-Lefschetz locus. Then, for a general element $X \in L$, there exists $\gamma \in H^{1,1}(X,\mb{Z})$ satisfying:
 $\gamma$ is the cohomology class of a semi-regular curve, say $C$, $H^1(\mo_X(-C)(d))=0$ and $\ov{\NL(\gamma)} \cong L$.
\end{cor}

\begin{rem}
 Before we prove this statement, we point out that the main difference between the above corollary and Proposition \ref{a81}
 is that we have the additional outcome that $H^1(\mo_X(-C)(d))=0$. If one looks at the diagrams in \ref{hf1} and Theorem \ref{hf12}
 one can see this condition precisely tells us that for every infinitesimal deformation of $C$,
 there exists a corresponding infinitesimal deformation of $X$ containing it. 
\end{rem}

\begin{proof}[Proof of Corollary \ref{a82}]
 By Proposition \ref{a81}, there exists a semi-regular curve $C$ and a surface $X$ containing it 
 such that $\ov{\NL([C])}$ is isomorphic 
 to $L$, where $[C] \in H^{1,1}(X,\mb{Z})$ is the cohomology class of $C$. 
 A lemma of Enriques-Severi-Zariski \cite[Corollary $7.8$]{R1} states that for $m \gg 0$, 
 $H^1(\mo_X(-C)(d-m))=0$, where $d$ is the degree of $X$. Replacing $C$ by a general
 element in the linear system, $|C+mH_X|$, we get the corollary.
\end{proof}

\begin{note}\label{dim2}
 Let $L$ be an irreducible component of the Noether-Lefschetz locus and $C$ be as in Proposition \ref{a81}. Denote by $P$ the Hilbert polynomial of $C$. Since $C$ is semi-regular,
 Theorem \ref{dim1} implies that there exists an irreducible component of $H_{P,Q_d}$, say $H_\gamma$ such that $H_{\gamma_{\red}}$ is an irreducible component of $H_{{P,Q_d}_{\red}}$ and $\pr_2(H_\gamma)_{\red}$
 is isomorphic to $L_{\red} \cong \ov{\NL(\gamma)}_{\red}$, where $L$ is as in the previous proposition. Denote by
 $L_\gamma:=\pr_1(H_\gamma)$ and by $T_CL_\gamma:=\pr_1(T_{(C,X)}H_\gamma)$. We say that $L_\gamma$ is \emph{non-reduced} 
 if $\dim L_\gamma<\dim T_CL_\gamma$.
\end{note}

We can then compute the dimension of $\ov{\NL(\gamma)}$ as follows:
\begin{prop}\label{dim}
Notations as in \ref{dim2}.  
The dimension of $\ov{\NL(\gamma)}$, \[\dim \ov{\NL(\gamma)}=\dim I_d(C)+\dim L_\gamma-h^0(\mathcal{O}_X(C)),\]
for a generic curve $C$ in $L_\gamma$.
\end{prop}

\begin{proof}
This follows from the fiber dimension theorem (see \cite[II Ex. $3.22$]{R1}) which states that for a morphism of finite type between two integral schemes, $f:X \to Y$,
$\dim X=\dim f^{-1}(y)+\dim Y$ for a general point $y \in Y$. 
We then have the following maps,
\[\pr_1:H_\gamma \twoheadrightarrow L_{\gamma}\, \, \, \pr_2:H_\gamma \twoheadrightarrow \ov{\NL(\gamma)}\]
where the generic fiber of $\pr_1$ is $\mb{P}(I_d(C))$ and that of $\pr_2$ is $\mathbb{P}(H^0(\mathcal{O}_X(C)))$.
We then conclude, \[\dim {H}_\gamma=\dim L_{\gamma}+\dim I_d(C)-1=\dim \ov{\NL(\gamma)}+h^0(\mo_X(C))-1.\]
Therefore,  $\dim \ov{\NL(\gamma)}=\dim L_{\gamma}+\dim I_d(C)-h^0(\mathcal{O}_X(C))$.
\end{proof}

We finally come to the main result of the section which tells us of a necessary and sufficient criterion for non-reducedness of irreducible components of $\NL_d$:
\begin{thm}[Non-reducedness]\label{a4}
Let $L$ be an irreducible component of the Noether-Lefschetz locus and $C$ be a semi-regular curve such that
$L \cong \ov{\NL([C])}$. Then, $L$ is non-reduced if and only if $L_\gamma$ is non-reduced in the sense of \ref{dim2}.
In particular, $L$ is non-reduced if and only if $H_{\gamma}$ is non-reduced, where $H_\gamma$ is as constructed in \ref{dim2}.
\end{thm}

\begin{proof}
Consider $C$ general in $L_\gamma$. Using the commutative diagram in Theomre \ref{hf12} we have that \[\Ima \rho_C \cap \Ima \beta_C = \rho_C \circ \pr_2(T_{(C,X)}H_{P,Q_d})=\beta_C \circ \pr_1(T_{(C,X)}H_{P,Q_d}).\]
By Corollary \ref{hf12c}, $T_X(\NL(\gamma))=\pr_2(T_{(C,X)}H_{P,Q_d})$. Therfore, Corollary \ref{dim4} implies that
\begin{eqnarray*}
 \dim T_X\NL(\gamma)&=& h^0(\mo_X(-C)(d)) + \dim T_{C}L_\gamma-h^0(\N_{C|X})\\
 &=&\dim I_d(C)+\dim T_CL_\gamma-h^0(\mo_X(C))
\end{eqnarray*}
where the last equality follows from Lemma \ref{dim3}.

Lemma \ref{dim} implies that \[\dim T_X\NL(\gamma) - \dim \NL(\gamma)  = \dim T_C(L_{\gamma})-\dim L_{\gamma}.\]
So, $\NL(\gamma)$ is non-reduced if and only if $L_\gamma$ is non-reduced.

Furthermore, using \ref{hf1} and the fact that the fiber over $C \in L_\gamma$ to the surjective projection morphism from $H_\gamma$ to $L_\gamma$ is $\mb{P}(I_d(C))$ one can see that \[\dim T_{(C,X)}H_\gamma=\dim T_CL_\gamma+h^0(\mo_X(-C)(d)), \, \, 
                                          \dim H_\gamma=\dim L_\gamma+\dim \mb{P}(I_d(C)),
                                         \]where $I_d(C)$ is the degree $d$ graded piece in the ideal of $C$.
Since $h^0(\mo_X(-C)(d))=\dim \mb{P}(I_d(C))$, $H_\gamma$ is non-reduced if and only if $T_CL_\gamma$ is non-reduced.
This proves the theorem.
\end{proof}

\section{Examples}

In this section we first look at some examples where a component of the Hilbert scheme is non-reduced. We then use these results in \S \ref{sum}
to give examples of a few non-reduced components of the Noether-Lefschetz locus.

\subsection{A smooth example}

\begin{thm}[Kleppe \cite{K1}]\label{a83}
There exists an irreducible component $L$ of the Hilbert scheme of curves in $\p3$ such 
that a general curve $C$ in $L$ is smooth, contained in a cubic surface and $h^1(\I_C(3)) \not=0$ for the following range:
\begin{enumerate}
 \item If $\deg(C) \ge 18$ then $g(C)>7+(\deg(C)-2)^2/8$.
  \item and $g(C)>-1+(\deg(C)^2-4)/8$ for $17 \ge \deg(C) \ge 14$
 \end{enumerate}
\end{thm}

\begin{rem}
 This is a generaization of an example due to Mumford. See \cite{Mu1} or \cite[\S $13$]{R3} for further details on his example.
 \end{rem}

\subsection{Martin-Deschamps and Perrin's example}\label{MD}


\begin{note}\label{a7}
 Let $a, d$ be positive integers, $d \ge 5$ and $a>0$. Let $X$ be a smooth projective surface in $\p3$ of degree $d$ containing a line $l$ and a smooth coplanar curve $C_1$ of degree $a$.
 Let $C$ be a divisor of the form $2l+C_1$ in $X$. Denote by $P$ the Hilbert polynomial of $C$.
\end{note}

\begin{thm}[Martin-Deschamps and Perrin \cite{mar1}]\label{a84}
There exists an irreducible component, say $L$ of $H_P$ such that a general curve $D \in L$ is a divisor in a smooth degree $d$ surface in $\p3$ of the form $2l'+C_1'$ where $l', C_1'$ are coplanar curves with $\deg(l')=1$
and $\deg(C_1')=a$. Furthermore, $L$ is non-reduced.
\end{thm}

\begin{proof}
 The theorem follows from \cite[Proposition $0.6$, Theorems $2.4, 3.1$]{mar1}.
\end{proof}


\subsection{Computing Castelnuovo-Mumford regularity}

\begin{para}
 In this section, we show that the Castelnuovo-Mumford regularity of the curves considered in Notation \ref{a7} is equal to $d$.
 \end{para}

\begin{note}\label{cas4}
Denote by $S$ the ring $\Gamma_*(\mo_{\p3})=\oplus_{n \in \mb{Z}} \Gamma(\p3,\mo_{\p3}(n))$.
 Let ${C_0}$ be a plane curve of degree $a+1$ for some positive integer $a$, defined by two equations, say $l_1$, $l_2G_1$ where $l_1, l_2$ are linear polynomials and $G_1$ is a smooth 
 polynomial of degree $a$. 
 Let $X$ be a smooth surface of degree $d$ containing ${C_0}$. Then $X$ is defined by an equation of the form $F_X:=F_1l_1+F_2l_2G_1$. By taking $X$ to be a general surface containing ${C_0}$, 
 we can assume that $G_1, F_1$ and $F_2$ are not contained
 in the ideal generated by $l_1, l_2$. Denote by $l$ the line defined by $l_1$ and $l_2$ and by $A_l$ its coordinate ring. 
 The aim of this section to prove that $\mo_X(-l) \otimes \mo_X(-{C_0})$ is $d$-regular.
 \end{note}
 
 We now recall a result on Castelnuovo-Mumford regularity which we will use later:
 \begin{thm}[{\cite[Ex. $4$E, Proposition $4.16$]{syz1}}]\label{syz1}
  Suppose that \[0 \to M' \to M \to M'' \to 0\]is a short exact sequence of finitely generated $S$-modules. Denote by $\mr{reg}(N)$ the Castelnuovo-Mumford regularity of a finitely generated $S$-module $N$.
  Then, $\mr{reg}(M') \le \max\{\mr{reg}(M),\mr{reg}(M'')-1\}, \mr{reg}(M'') \le \max\{\mr{reg}(M),\mr{reg}(M')+1\}, \mr{reg}(M) \le \max\{\mr{reg}(M'),\mr{reg}(M'')\}.$ Furthermore,
  given a finitely generated $S$-module $M$, $\mr{reg}(M) \ge \mr{reg}(\widetilde{M})$ where $\widetilde{M}$ is the associated coherent sheaf.  
 \end{thm}

 \begin{para}
 Denote by $I_{C_0}$ (resp. $I_X$, $I_{l}$) the ideal of ${C_0}$ (resp. $X, l$). 
 Then we have the following commutative diagram:
 \begin{equation}\label{com9}
\begin{diagram}
    0&\rTo&I_X&\rTo^{f_1}&I_{C_0}&\rTo^{f_2}&I_{C_0}/I_X&\rTo&0\\
    & &\uTo^{\cong}&\circlearrowleft&\uOnto^{f_4}& \\
    & &S(-d)&\rTo^{f_3}&S(-1) \oplus S(-a-1)& \\
    & & & &\uInto^{f_5}&\\ 
    & & & &S(-2-a)&\\
   \end{diagram}
   \end{equation}
where $f_1$ is the natural inclusion map, $f_2$ is the quotient map, \[f_3:G \mapsto (GF_1,GF_2),\, \, f_4:(H_1, H_2) \mapsto H_1l_1+H_2l_2G_1,\, \, f_5:G \mapsto (Gl_2G_1,-Gl_1).\]
Note that the top horizontal row and the rightmost vertical sequence are exact. 
\end{para}

\begin{lem}\label{cas6}
 The kernel of the natural morphism $f'_5+f'_3$ from   $A_{l}(-2-a) \oplus A_{l}(-d)$ to $A_{l}(-1)\oplus A_{l}(-a-1)$ which
 maps $(\ov{G},\ov{H})$ to $(\ov{HF_1+Gl_2G_1},\ov{HF_2-Gl_1})$ is $A_{l}(-a-2) \oplus 0$, where $G, H \in S(-a-1), S(-d)$, 
 respectively and $\ov{G},\ov{H}$ are  its images in $A_l(-2-a)$ and $A_l(-d)$, respectively
\end{lem}

\begin{proof}
 Since $I_{l}$ is generated by $l_1$ and $l_2$, $\ov{HF_1+Gl_2G_1}$ and $\ov{HF_2-Gl_1}$ are both 
 zero if and only if both $HF_1$ and $HF_2$ are both in $I_{l}$. Since $F_1, F_2$ are not contained in $I_{l}$, by assumption,
 and $I_{l}$ is a prime ideal, this implies $H$ is contained in $I_{l}$. In other words, the kernel of the map $f'_5+f'_3$ is isomorphic to pairs $(\ov{G},0)$ where $G \in S(-a-2)$. 
 This proves the lemma.
\end{proof}

\begin{prop}\label{cas8}
$A_{l} \otimes_S I_{C_0}/I_X$ is $d+1$-regular.
\end{prop}
 
 \begin{proof}
 Using the exactness of the sequences in the diagram (\ref{com9}), we can conclude that $\Ima f_3+\Ima f_5$ is contained in $\ker(f_2 \circ f_4)$. We now show the converse inclusion. Let $h \in \ker(f_2 \circ f_4)$.
 So, $f_4(h) \in \Ima f_1$, which implies that there exists $b \in S(-d)$ such that $f_4(h-f_3(b))=0$. Therefore, there exists $c \in S(-2-a)$ such that $f_5(c)=h-f_3(b)$. This proves that 
 $\ker(f_2 \circ f_4) \subset \Ima f_3+\Ima f_5$, hence equality.
 
 Note that the image of $f_5$ maps to zero under $f_4$ due to exactness of the sequence. But, $f_4 \circ f_3$ is injective by the commutative of the above diagram and the injectivity of $f_1$. 
 So, $\Ima f_3 \cap \Ima f_5=0$. So, $\ker(f_4 \circ f_2)=S(2-a) \oplus S(-d)$.
 So, we have the following short exact sequence of $S$-modules:
 \begin{equation}\label{eqcas1}
  0 \to S(-2-a) \oplus S(-d) \xrightarrow{f_5\oplus f_3} S(-1)\oplus S(-a-1) \xrightarrow{f_4 \circ f_2} I_{C_0}/I_X \to 0
  \end{equation}
  Since tensor product is right exact, using Lemma \ref{cas6} we get the following exact sequence:
  \[0 \to A_{l}(-d) \xrightarrow{f'_3} A_{l}(-1)\oplus A_{l}(-a-1) \xrightarrow{f'_4 \circ f'_2} A_{l} \otimes_S I_{C_0}/I_X \to 0\]
  where $f'_i$ are induced by $f_i$ for $i=1,...,5$. 
  
  Using the short exact sequence \[0 \to S(-2) \to S(-1) \oplus S(-1) \to A_l \to 0\] we see that the Castelnuovo-Mumford
  regularity of $A_{l}(-d)$ and $A_{l}(-1)\oplus A_{l}(-a-1)$ is less than or equal to $d$.
  Finally, Theorem \ref{syz1} implies that the Castelnuovo-Mumford regularity of $A_{l} \otimes_S I_{C_0}/I_X$ is at 
  most $d+1$.
 \end{proof}

\begin{thm}\label{cas12}
 The sheaf $\mo_X(-l-{C_0})$ is $d$-regular. 
\end{thm}

\begin{proof}
 Consider the natural surjective morphim $A_X \to A_l$.
 Tensoring this by $I_{C_0}/I_X$, we obtain the short exact sequence,
 \[0 \to \ker(p) \to I_{C_0}/I_X \xrightarrow{p} A_{l} \otimes_S I_{C_0}/I_X \to 0.\]
 Using Proposition \ref{cas8} and Theorem \ref{syz1} we conclude that $\widetilde{\ker(p)}$ is $d$-regular. 
 It remains to prove that $\widetilde{\ker{p}} \cong \mo_X(-C_0-l)$.

 Since $\Gamma_*(\mo_X(-C_0))=I_{C_0}/I_X$ (by definition) and $\Gamma_* \mo_l=A_l$, \cite[Proposition II.$5.15$]{R1}
 implies that $\widetilde{I_{C_0}/I_X}$ and $\widetilde{A_l}$ is isomorphic to $\mo_X(-C_0)$ and $\mo_l$, respectively. 
 Now, the associated module functor $\widetilde{ }$  is exact and commutes with tensor product (\cite[Proposition II.$5.2$]{R1}).
 Applying this functor to the last short exact sequence we get,
 \[0 \to \widetilde{\ker(p)} \to \mo_X(-C_0) \xrightarrow{\widetilde{p}} \mo_l \otimes \mo_X(-C_0) \to 0\]
 where $\widetilde{p}$ arises from tensoring by $\mo_X(-C_0)$ the natural surjective morphism
 $\mo_X \to \mo_l$.
 
 Consider now the short exact sequence \[0 \to \mo_X(-l) \to \mo_X \to \mo_l \to 0.\]
 Since $\mo_X(-C_0)$ is a flat $\mo_X$-module, we get the short exact sequence,
 \[0 \to \mo_X(-C_0-l) \to \mo_X(-C_0) \xrightarrow{\widetilde{p}} \mo_l \otimes_{\mo_X} \mo_X(-C_0) \to 0.\]
 By the universal property of the kernel, $\widetilde{\ker{p}}$ is isomorphic to $\mo_X(-C_0-l)$.
 Hence, $\mo_X(-C_0-l)$ is $d$-regular.
\end{proof}

\subsection{Examples of Non-reduced components of $\NL_d$}\label{sum}

\begin{para}
 Before we come to the final result of this article we recall a result by Kleiman and Altman which tells us given a curve when does there exist a \emph{smooth} surface in $\p3$ containing it.
 This will be used to prove the existence of certain components of the Noether-Lefschetz locus. We then prove non-reducedness.
\end{para}

\begin{note}
Let $C$ be a projective curve in $\p3$. Denote by \[D_e:=\{x \in C| \dim \Omega^1_{C,x}=e\}.\]The theorem in this case states that,
 \begin{thm}[{\cite[Theorem $7$]{kleim}}]\label{exi1}
  If for any $e>0$ suc
  that $D_e \not= \emptyset$ we have that $\dim D_e+e$ is less than $3$ then there exists a smooth surface in $\p3$ containing $C$. Moreover, if $C$ is $d-1$-regular then
 there exists a smooth degree $d$ surface containing $C$.
 \end{thm}
\end{note}

\begin{para}
We now collect the previous results to give some examples of non-reduced components of $\NL_d$. In the following theorem the phrase \emph{C is general}
would mean that the point on the Hilbert scheme corresponding to the curve $C$ is a smooth point which implies that no further irreducible component of the Hilbert scheme pass through this point.
\end{para}

\begin{thm}\label{a71}
The following statements are true:
\begin{enumerate}
 \item Let $C$ be a smooth curve or a curve with at most double points 
 (i.e., points $x \in C_{\red}$ such that $\dim \Omega^1_{C_{\red},x}=2$). Take $d>\deg(C)+4$. Then there exists smooth 
 degree $d$ surfaces in $\p3$ containing $C$.
 \item Let $C$ be a curve and $C \subset X$, a degree $d$ surface. Let $\gamma=[C]$. Then $\NL(\gamma)$ is non-reduced when 
\begin{enumerate}[(a)]
\item $C$ is a general element in $L$ as in Theorem \ref{a83} and $d \ge  \deg(C)+4$.
\item $C$ is a general element in $L$ is as in Theorem \ref{a84} and $\ov{\NL(\gamma)}_{\red}$ coincides with $\pr_2(\pr^{-1}(L))_{\red}$ where $P$ is the Hilbert polynomial of $C$ and $\pr_i$ are natural projection
morphisms from $H_{P,Q_d}$ to the respective components.
\end{enumerate} 
\end{enumerate}
\end{thm}

\begin{proof}
\begin{enumerate}
\item It follows from the proof of Lemma \ref{hf11} that $C_{\red}$ is $d-1$ regular. Using Theorem \ref{exi1}, we can then conclude that there exists smooth degree $d$ surfaces containing $C$.
\item Note that for all $C$ as in $(a)$ and $(b)$, $C$ is either smooth or $C_{\red}$ has at most double points.
Using $(1)$ there exists smooth surfaces $X$ in $\p3$ containing such $C$ of degree $d$. Take $\gamma=[C]$.
We now prove non-reducedness. 
\begin{enumerate}[(a)]
 \item Lemma \ref{hf11} tells us that $C$ is semi-regular and $h^1(\mo_X(-C)(d))=0$. Denote by $P$ the Hilbert polynomial of $C$.
 Theorem \ref{dim1} implies that there exists an irreducible components $H_\gamma$ in $H_{P,Q_d}$ such that 
 $\pr_2(H_\gamma)_{\red}=\ov{\NL(\gamma)}_{\red}$. Since $C$ is general in $L$, $\pr_1(H_\gamma)_{\red}=L_{\red}$.
 Using Corollary \ref{dim4} we notice that the natural morphim from $H^0(\N_{X|\p3})$ to $H^0(\N_{X|\p3} \otimes \mo_C)$ is surjective and 
 $T_CL_\gamma=H^0(\N_{C|\p3})$, where $L_\gamma=\pr_1(H_\gamma)$ and $T_CL_\gamma$ is as defined in \ref{dim2}. 
 Theorem \ref{a83} states that $\dim L_\gamma < h^0(\N_{C|\p3})$ for $C$ in $L_\gamma$, general.
 Finally, Theorem \ref{a4} shows that $\ov{\NL(\gamma)}$ is non-reduced.
 \item Using Theorem \ref{cas12}, we have that $h^1(\mo_X(-C)(d))=0$, hence Corollary \ref{dim4} implies the natural morphism from $H^0(\N_{X|\p3})$ to $H^0(\N_{X|\p3}  \otimes \mo_C)$ is surjective. 
 Following exactly as in the proof of Theorem \ref{dim} by replacing $H_\gamma$ by an irreducible component in $\pr_1^{-1}(L)$ such that $\pr_2(H_\gamma)_{\red} \cong \ov{\NL(\gamma)}$, we get 
 after using Lemma \ref{a4e} that $\dim \NL(\gamma)=\dim I_d(C)-1+\dim L$.
 Using Corollary \ref{hf12c} and Lemma \ref{a4e} we get 
 $\dim T_X(\NL(\gamma)) \ge \dim I_d(C)-1+h^0(\N_{C|\p3}).$
Theorem \ref{a84} implies $\dim L<h^0(\N_{C|\p3})$, hence non-reducedness of $\ov{\NL(\gamma)}$.
\end{enumerate}
\end{enumerate}
\end{proof}

We give an example satisfying the conditions in Theorem \ref{a71}$2$(b).

\begin{thm}
 Let $d \ge 5$, $X$ be a smooth degree $d$ surface containing $2$ coplanar lines, $l_1, l_2$. Let $C$ be a divisor in $X$ of the form $2l_1+l_2$, $\gamma$ the cohomology class of $C$.
 Then, $\ov{\NL(\gamma)}_{\red}$ is isomorphic to $\ov{(\NL([l_1]) \cap \NL([l_2]))}_{\red}$. In particular, $\ov{\NL(\gamma)}$ is non-reduced.
\end{thm}

\begin{proof} 
 Note that $\NL(\gamma)$ contains the space of smooth degree $d$ surfaces containing two coplanar lines. But this space is of codimension $2d-6$. So,
 $\codim \ov{\NL(\gamma)} \le 2d-6$. If $\codim \ov{\NL(\gamma)} < 2d-6$ \cite[Proposition $1.1$]{v3} implies that $\ov{\NL(\gamma)}$ is reduced and either it parametrizes smooth degree $d$ 
 surfaces containing a line or containing a conic. This means that there exists $\gamma'$ a class of a line or a conic in $X$ such that $\NL(\gamma)=\NL(\gamma')$.
 Note that if $\gamma'_{\prim}$ is a multiple of $[l_1]_{\prim}$ or $[l_2]_{\prim}$ or $[l_1+l_2]_{\prim}$ then this locus parametrizes surfaces such that the cohomology class of $l_1$ and $l_2$ remains of type $(1,1)$.
 Lemma \ref{hf11} tells us $l_1, l_2$ are semi-regular for $d \ge 5$. Theorem \ref{dim1} then implies this space parametrizes surfaces containing $2$ coplanar lines, hence $\codim \ov{\NL(\gamma)} = 2d-6$, which gives us a contradiction.
 
 Choose $X$ in $\NL(\gamma)$ such that if $H$ is the hyperplane containing $l_1 \cup l_2$ then $H \cap X=l_1 \cup l_2 \cup D$ with $D$ irreducible. Note that such $X$ exists and the new $\gamma$ corresponding
 to the cohomology class of the divisor $2l_1+l_2$ in this surface defines the same component $\ov{\NL(\gamma)}$.
 Denote by $E$ the preimage in $H^0(\mo_{\p3}(d))$ of the vector space $T_X(\NL(\gamma)) \subset H^0(\mo_X(d))$. For an integer $k$, denote by \[E_k:=[E:S_{d-k}]:=\{P \in S_k|P.S_{d-k} \subset E\},\] where 
 $S_{d-k}$ is $H^0(\mo_{\p3}(d-k))$. Now, \cite[$4.a.5$]{GH} implies that $I(l_1 \cup l_2)$ is contained in $\oplus_k E_k$. 
  Now, \cite[Proposition $4$]{c1} and the argument after \cite[Proposition $5$]{c1} tells us that $\gamma'=a_1[l_1]+a_2[l_2]+b[H_X]$, where $\gamma'$ is the class $[C']$ of a line or a conic
  and $a_1, a_2, a_3$ are rationals. 
  
  By the above argument, $C'$ is not $l_1, l_2$ or $l_1 \cup l_2$. Denote by $t_0$ and $t_1$ the integers $l_1.C'$ and $l_2.C'$, respectively. 
  Then, intersecting the equality $[C']=a_1[l_1]+a_2[l_2]+b[H_X]$ by $H_X, l_1, l_2$ and $C'$, respectively, we have
  \begin{eqnarray}
  \deg(C')&=&a_1+a_2+bd\\
  t_0&=&a_1(2-d)+a_2+b\\
  t_1&=&a_1+a_2(2-d)+b\\
  C'^2&=&a_1t_0+a_2t_1+\deg(C')b
  \end{eqnarray}
  Assume that $l_i\not| C'$. Using adjunction formula one can check $C'^2$ is $2-d$ if $C'$ is a line and $6-2d$ if $C'$ is a conic.
  Then, using any standard programming language (for eg. Maple), one can see that there does not exist a solution to these set of equations.
  
  The only case that remains is when $C'$ is a conic and of the form $l' \cup l_i$ for $i=1$ or $2$ and $l'$ is distinct from $l_1, l_2$. We can then replace in the above equation $C'$ by $l'$ and replace 
  $a_1$ (resp. $a_2$) by $a_1-1$ (resp. $a_2-1$) if $i=1$ (resp. $i=2$). Then the above result tells us again that there are no solutions. So, $\codim \ov{\NL(\gamma)}$ has to be $2d-6$.
  \end{proof}

\bibliographystyle{alpha}
 \bibliography{researchbib}
 
\end{document}